\documentclass[reqno,a4paper]{amsart}
\usepackage{enumitem}

\setenumerate{label=\textnormal{(\arabic*)}}

\usepackage{amsmath,amssymb,dsfont,verbatim,mathtools,bm,geometry,fge,mathrsfs}

\usepackage[raggedright]{titlesec}
\usepackage{booktabs}

\usepackage[latin1]{inputenc}
\usepackage[raggedright]{titlesec}
\usepackage{mathtools}
\usepackage{tikz}
\usepackage{float,subfig}
\usepackage{amsrefs}
\usepackage{algorithm2e}
\usepackage[foot]{amsaddr}

\titleformat{\chapter}[display]
{\normalfont\huge\bfseries}{\chaptertitlename\\thechapter}{20pt}{\Huge}
\titleformat{\section}
{\normalfont\Large\bfseries\center}{\thesection}{1em}{}
\titleformat{\subsection}
{\normalfont\large\bfseries}{\thesubsection}{1em}{}
\titleformat{\subsubsection}[runin]
{\normalfont\normalsize\bfseries}{\thesubsubsection}{1em}{}
\titleformat{\paragraph}[runin]
{\normalfont\normalsize\bfseries}{\theparagraph}{1em}{}
\titleformat{\subparagraph}[runin]
{\normalfont\normalsize\bfseries}{\thesubparagraph}{1em}{}
\titlespacing*{\chapter} {0pt}{50pt}{40pt}
\titlespacing*{\section} {0pt}{3.5ex plus 1ex minus .2ex}{2.3ex plus .2ex}
\titlespacing*{\subsection} {0pt}{3.25ex plus 1ex minus .2ex}{1.5ex plus .2ex}
\titlespacing*{\subsubsection}{0pt}{3.25ex plus 1ex minus .2ex}{1.5ex plus .2ex}
\titlespacing*{\paragraph} {0pt}{3.25ex plus 1ex minus .2ex}{1em}
\titlespacing*{\subparagraph} {\parindent}{3.25ex plus 1ex minus .2ex}{1em}

\usepackage{titletoc}
\contentsmargin{2.55em}
\dottedcontents{section}[1.5em]{}{2.8em}{1pc}
\dottedcontents{subsection}[3.7em]{}{3.4em}{1pc}
\dottedcontents{subsubsection}[7.6em]{}{3.2em}{1pc}
\dottedcontents{paragraph}[10.3em]{}{3.2em}{1pc}


\newtheorem{theorem}{Theorem}[section]
\newtheorem{lemma}[theorem]{Lemma}
\newtheorem{proposition}[theorem]{Proposition}
\newtheorem{corollary}[theorem]{Corollary}

\theoremstyle{definition}
\newtheorem{definition}[theorem]{Definition}

\newtheorem{notation}[theorem]{Notation}
\newtheorem{example}[theorem]{Example}

\theoremstyle{remark}
\newtheorem{remark}[theorem]{Remark}

\DeclareMathOperator{\HH}{H}
\DeclareMathOperator{\Supp}{Supp}

\DeclareMathOperator{\en}{en}

\DeclareMathOperator{\factors}{factors}

\DeclareMathOperator{\sgn}{sgn}
\DeclareMathOperator{\st}{st}
\DeclareMathOperator{\ev}{ev}

\DeclareMathOperator{\dir}{dir}

\DeclareMathOperator{\Dir}{Dir}
\DeclareMathOperator{\Ima}{Im}
\DeclareMathOperator{\Coef}{Coef}

\DeclareMathOperator{\lcm}{lcm}
\DeclareMathOperator{\mult}{mult}
\DeclareMathOperator{\gap}{gap}
\newcommand{\ov}{\overline}

\begin{document}

\title[The two-dimensional Jacobian conjecture]{The two-dimensional Jacobian conjecture and the lower side of the Newton polygon}

\author[Jorge A. Guccione]{Jorge A. Guccione$^{1,}$$^2$}
\address{$^1$ Universidad de Buenos Aires. Facultad de Ciencias Exactas y Naturales. Departamento de Matem\'atica. Buenos Aires. Argentina}
\address{$^2$ CONICET-Universidad de Buenos Aires. Instituto de Investigaciones Matem\'aticas ``Luis A. santal\'o'' (IMAS). Buenos Aires. Argentina}


\email{vander@dm.uba.ar}

\author[Juan J. Guccione]{Juan J. Guccione$^{1,}$$^3$}
\address{$^3$ CONICET. Instituto Argentino de Matem\'atica (IAM). Buenos Aires. Argentina}


\email{jjgucci@dm.uba.ar}

\thanks{Jorge A. Guccione and Juan J. Guccione were supported by UBACyT 20020150100153BA (UBA) and PIP 11220110100800CO (CONICET)}

\author{Christian Valqui$^{4,}$$^5$}
\address{$^4$Pontificia Universidad Cat\'olica del Per\'u, Secci\'on Matem\'aticas, PUCP,
Av. Universitaria 1801, San Miguel, Lima 32, Per\'u.}

\address{$^5$Instituto de Matem\'atica y Ciencias Afines (IMCA) Calle Los Bi\'ologos 245. Urb San C\'esar.
La Molina, Lima 12, Per\'u.}
\email{cvalqui@pucp.edu.pe}
\thanks{Christian Valqui was supported by PUCP-DGI-CAP-2015-185.}

\begin{abstract}
We prove that if the Jacobian Conjecture in two variables is false and $(P,Q)$ is a standard minimal pair, then the Newton polygon $\HH(P)$ of $P$ must satisfy several restrictions that had not been found previously. This allows us to discard some of the corners found in~\cite{GGV}*{Remark~7.9} for $\HH(P)$, together with some of the infinite families found in~\cite{H}*{Theorem~2.25}.
\end{abstract}

\maketitle


\section*{Introduction}

Let $K$ be a characteristic zero field and let $L\coloneqq K[x,y]$ be the polynomial algebra in two indeterminates. The Jacobian Conjecture in dimension two, stated by Keller in~\cite{K}, says that any pair of polynomials $P,Q\in L$ with $[P,Q]\coloneqq \partial_x P \partial_y Q - \partial_x Q \partial_y P\in K^{\times}$ defines an automorphism of $L$. If this conjecture is false, then the degrees of the components $P\coloneqq f(x)$, $Q\coloneqq f(y)$ of the hypothetical counterexample $f$ satisfy certain restrictions, found by various authors. In~\cite{M} the author finds that the only possible pairs $(\deg(P),\deg(Q))$ with both entries lower than $100$ are $(64,48)$, $(50,75)$, $(56,84)$ and $(66,99)$. Then he discards these four cases by hand. In~\cite{H}*{Theorem~2.25} Heitmann determines for  possible pairs $(\deg(P),\deg(Q))$ various families of the form $(a+bj,c+dj)$ with $a,b,c,d\in\mathds{N}$ and $j$ running on $\mathds{N}$, and confirms the four pairs found by Moh. For example, the family $16(1+2j,1+3j)$ yields for $j=1$ the first case found by Moh.

But Heitmann determines not only the degrees, he also says something about the shape of the support of $P$ and $Q$ in a hypothetical counterexample. He associates to each counterexample a corner and gives a list of small possible corners in~\cite{H}*{Theorem~2.24}. For example, the corner $(4,12)$ together with the family $(1+2j,1+3j)$ yields the family of degrees mentioned above. The same possible corners of~\cite{H} were confirmed in~\cite{GGV}*{Remark~7.9}, using more elementary methods and discrete geometry on the plane. The shape of the support of $P$ was described in more detail, finding an edge with starting point $A'$ below the main diagonal and end point $A$ above the diagonal (See for example the beginning of section 6 in~\cite{GGV}). Various restriction were found for $A_0\coloneqq \frac{1}{m} A$ where $m:=\deg(P)/\gcd(\deg(P),\deg(Q))$, leading to the list of~\cite{GGV}*{Remark~7.9}, which contains the list of~\cite{H}*{Theorem~2.24}.

In the present paper we focus on restrictions for $A_0'\coloneqq \frac{1}{m} A'$, which allows us to discard some of the corners found in~\cite{GGV}*{Remark~7.9}, together with some of the infinite families found in~\cite{H}*{Theorem~2.25}. We will see in Theorem~\ref{condicion principal} that for such an $A_0'$ there must exist a direction $(\rho,\sigma)$ and $(\rho,\sigma)$-homogeneous elements $G,R\in L$ such that
\begin{equation}\label{restriccion}
A_0'=\en_{\rho,\sigma}(R)\quad\text{and}\quad [G,R]= R^2.
\end{equation}
Here we adopt notations and definitions from~\cite{GGV}*{Section~1}. The condition~\eqref{restriccion} is the main restriction that allows to discard as possible $A_0'$ all the points for which such $(\rho,\sigma)$, $G$ and $R$ do not exist. Our findings can be resumed in the following three useful results:

\begin{itemize}

\smallskip

\item[-] We cannot discard any point below the line $2y=x$ as possible $A_0'$ (Proposition~\ref{ejemplo}).

\smallskip

\item[-] If $(a',b')$ is a possible $A_0'$, then $b'\le (a'-b'-1)^2$ (Proposition~\ref{cota}).

\smallskip

\item[-] No possible $A_0'$ can be of the form $\wp(n'+1,n')$ for $\wp,n'\in\mathds{N}$ (Proposition~\ref{casos imposibles}).

\smallskip
\end{itemize}
In order to obtain the last two conditions we use our main technical result, Proposition~\ref{finitas direcciones}, which yields restrictions on the directions $(\rho,\sigma)$ that can occur for an $R$ as above if you fix the starting point. This result also allows to write an algorithm to determine all possible $A_0'$ with $v_{1,-1}(A_0')<N$ for some fixed $N$  (by Corollary~\ref{cota} there are only finitely many), which we will do in a future article.

A straightforward computation shows that for any $(a,b)$ with $a>b$, the $(1,-1)$-homogeneous elements $R\coloneqq x^{a-b}(1+xy)^b$ and $G\coloneqq -\frac{1}{a-b}xy R$ satisfy
$$
(a,b)=\en_{1,-1}(R)\quad\text{and}\quad [G,R]= R^2.
$$
Hence, in order to obtain the restrictions for $A_0'$ we need the following result of~\cite{CN}: The support of a component of a Jacobian pair cannot have an edge with slope~$1$. In the first section we generalize this result, using the more elementary proof of Makar Limanov in~\cite{ML}.  In the second section, assuming that the Jacobian conjecture is false, we take a Jacobian pair $(P,Q)$ that is a minimal pair and a standard $(m,n)$-pair with $P(0,0)Q(0,0)\ne 0$ and we study the lower part of $\frac{1}{m}H(P)$, where $H(P)$ is the Newton polygon of $P$. In the third section we introduce the notion of admissible chain, which encodes in an abstract way the properties of the lower part of $\frac{1}{m}H(P)$.

\section{Generalizing Cassier-Nouges}

Throughout this paper $l$ denotes a fixed natural number. Let $K$ be characteristic zero field. Given a $K$-algebra $A$ and $\eta\in A$, we let
$\ev_{\eta}\colon A[y]\to A$ denote the {\em evaluation map} $\ev_{\eta}(P)\coloneqq P(\eta)$. For the sake of brevity, we set $L^{(l)}\coloneqq K[x^{\pm \frac{1}{l}},y]$, $\hat L_-^{(l)}\coloneqq K((x^{-1/l}))[y]$ and $\hat L_+^{(l)}\coloneqq K((x^{1/l}))[y]$. Let~$P=\sum a_{\frac{i}{l},j} x^{\frac{i}{l}} y^j \in \hat L_+^{(l)}\cup \hat L_-^{(l)}$. By definition the {\em coefficient} at $x^{\frac{i}{l}}y^j$ of $P$ and the {\em support} of~$P$ are
$$
\Coef_{x^{\frac{i}{l}} y^j}\coloneqq a_{\frac{i}{l},j}\quad\text{and}\quad \Supp(P) \coloneqq  \left\{\left(i/l,j\right): a_{\frac{i}{l},j}\ne 0\right\},
$$
respectively. We call a pair $(\rho,\sigma)\in \mathds{Z}^2$ a {\em direction} if $\gcd(\rho,\sigma)=1$. We will denote by $\mathfrak{V}$ the set of directions. For $(\rho,\sigma)\in \mathfrak{V}$ and $P\in \hat L_-^{(l)}\cup \hat L_+^{(l)}$, we define the {\em $(\rho,\sigma)$-degree} of $P$ as
$$
v_{\rho,\sigma}(P)\coloneqq  \sup\Bigl\{\frac{i}{l} \rho +j\sigma : (i/l,j) \in \Supp(P)\Bigr\}\in \mathds{R}\cup\{\infty\}.
$$
Note that if $\rho>0$, then $v_{\rho,\sigma}(P)<\infty$ if and only if $P\in \hat L_-^{(l)}$, while if $\rho<0$, then $v_{\rho,\sigma}(P)<\infty$ if and only if $P\in \hat L_+^{(l)}$. For $P$ satisfying $v_{\rho,\sigma}(P)<\infty$, we define the {\em $(\rho,\sigma)$-leading term} of $P$ as
$$
\ell_{\rho,\sigma}(P)\coloneqq   \sum_{\{\rho \frac{i}{l} + \sigma j = v_{\rho,\sigma}(P)\}} a_{\frac{i}{l},j} x^{\frac{i}{l}} y^j.
$$
We define the set of {\em directions associated with} $P$ as
$$
\Dir(P)\coloneqq \{(\rho,\sigma)\in\mathfrak{V}:v_{\rho,\sigma}(P)<\infty\text{ and }\#\Supp(\ell_{\rho,\sigma}(P))>1\}.
$$
Let $\eta\in K((x^{-1/l}))$ and $P\in \hat L_-^{(l)}$ or $\eta\in K((x^{1/l}))$ and $P\in \hat L_+^{(l)}$. A straightforward computation shows that if $v_{\rho,\sigma}(\eta)=v_{\rho,\sigma}(y)=\sigma$ and $v_{\rho,\sigma}(P)<\infty$, then
\begin{equation}\label{evaluacion en el borde}
\ev_{\ell_{\rho,\sigma}(\eta)}(\ell_{\rho,\sigma}(P))=\ell_{\rho,\sigma}(\ev_{\eta}(P)),
\end{equation}
whenever the left hand side is nonzero. The algebras $\hat L_-^{(l)}$ and $\hat L_+^{(l)}$ are topological algebras in a natural way. A local base at zero of $\hat L_-^{(l)}$ is the family $(V_u)_{u\in \mathds{Z}}$, where
$$
V_u \coloneqq  \left\{P\in \hat L_-^{(l)}: \text{ if $(i/l,j)\in \Supp(P)$, then $i\le u$} \right\},
$$
while a local base at zero of $\hat L_+^{(l)}$ is the family $(V_u)_{u\in \mathds{Z}}$, where
$$
V_u \coloneqq  \left\{P\in \hat L_+^{(l)}: \text{ if $(i/l,j)\in \Supp(P)$, then $i\ge u$} \right\}.
$$
There is a unique isomorphism of topological algebras $\varphi\colon \hat L_+^{(l)}\to \hat L_-^{(l)}$ such that $\varphi(x^{1/l}) = x^{-1/l}$ and $\varphi(y) = y$. We define two continuous derivations $D_x$ and $D_y$ on $\hat L_-^{(l)}$ and $\hat L_+^{(l)}$, by
$$
D_x(x^{1/l})\coloneqq \frac 1l x^{\frac 1l -1},\qquad D_x(y)\coloneqq 0,\qquad D_y(x^{1/l})\coloneqq 0\qquad\text{and}\qquad D_y(y)\coloneqq 1.
$$
It is easy to check that
$$
\Ima(D_x)=\left\{P =\sum a_{\frac{i}{l},j} x^{\frac{i}{l}} y^j : a_{-1,j} = 0 \text{ for all } j\right\}.
$$
Moreover, a direct computation shows that
\begin{equation}\label{derivada total}
D_x(\ev_{\eta}(P))=\ev_{\eta}(D_x(P))+\ev_{\eta}(D_y(P))D_x(\eta).
\end{equation}
We also define continuous linear maps $\int_y$ on $\hat L_-^{(l)}$ and $\hat L_+^{(l)}$, and $\int_x$ on $\Ima(D_x)$, by
$$
\int_y x^{\frac{i}{l}} y^j\coloneqq \frac{1}{j+1}x^{\frac{i}{l}} y^{j+1}\qquad\text{and}\qquad \int_x x^{\frac{i}{l}} y^j\coloneqq \frac{1}{i/l+1}x^{\frac{i}{l}+1} y^j.
$$

\begin{lemma}\label{forma exacta}
Let $g, f\in \hat L_-^{(l)}$. Assume that $g\in\Ima(D_x)$. If $D_y(g)=D_x(f)$, then the differential form $gdx+fdy$ is exact; i.e., there exists $H\in \hat L_-^{(l)}$ such that $D_y(H)=f$ and $D_x(H)=g$.
\end{lemma}
\begin{proof}
Consider
$$
H\coloneqq \int_y f+\int_x\left(g- \int_y D_y(g)\right).
$$
A direct computation shows that $D_y(H)=f$ and $D_x(H)=g$.
\end{proof}

\begin{proposition}\label{puiseux} Let $P\in L^{(l)}$ and $n\coloneqq v_{0,1}(P)$. There exist $u\in K[x^{\frac 1l},x^{-\frac 1l}]$, $r\in\mathds{N}$, and a family $(\eta_i)_{1\le i\le n}$ of elements of $K((x^{-1/r}))$, such that
\begin{equation}\label{producto}
P=u \prod_{i=1}^{n}(y-\eta_i).
\end{equation}
\end{proposition}

\begin{proof} See~\cite{E}*{Corollary~13.15, page~295}.
\end{proof}

\begin{remark} The same result holds with each $\eta_i$ in $K((x^{1/r}))$ instead of $K((x^{-1/r}))$.
\end{remark}

\begin{proposition}\label{Newton Puiseux} Let $P\in L^{(l)}$ and $(\rho,\sigma)\in \Dir(P)$. Assume $\rho\ne 0$ and let $\sgn(\rho)$ denote the
sign of $\rho$. There exists $r\in \mathds{N}$ and $\eta\in K((x^{-\frac{\sgn(\rho)}{r}}))$ of the form
$$
\eta=\lambda x^{\sigma/\rho} + \sum_{t\rho<r\sigma} \lambda_t x^{t/r},
$$
with $\lambda\in K^{\times}$, such that $\ev_{\eta}(P)=0$. Moreover $y-\lambda x^{\sigma/\rho}$ divides $\ell_{\rho,\sigma}(P)$.
\end{proposition}

\begin{proof} We assume that $\rho>0$ and leave the case $\rho<0$, which is similar, to the reader. Let $n$, $r$, $u$ and $\eta_i$ be as in Proposition~\ref{puiseux}. Set $P_i\coloneqq y-\eta_i\in\hat L^{(r)}_-$. Since $\ell_{\rho,\sigma}$ is multiplicative, from~\eqref{producto} we obtain
$$
\ell_{\rho,\sigma}(P)=\ell_{\rho,\sigma}(u) \prod_{i=1}^{n}\ell_{\rho,\sigma}(P_i).
$$
Note that $\ell_{\rho,\sigma}(u)$ is a monomial, because all the monomials in $x^{\pm 1/r}$ have different $(\rho,\sigma)$-degrees. Since $(\rho,\sigma)\in \Dir(P)$, there exists at least one $i$, say $i_0$, such that the factor $\ell_{\rho,\sigma}(P_{i_0})$ has two terms, which necessarily are $y$ and $\lambda x^{\sigma/\rho}$ for some $\lambda\in K^{\times}$.  The result follows immediately taking $\eta=\eta_{i_0}$.
\end{proof}

\begin{theorem}\label{cassier nouges generalizado} Let $P,Q\in L^{(l)}$. Set
$$
\tilde J=\tilde J(P,Q)\coloneqq \Coef_{x^{-1}y^0}(QD_x(P))x^{-1}+\int_y [P,Q],
$$
and write $g\coloneqq \tilde J-Q D_x(P)$. The following facts hold:
\begin{enumerate}

\smallskip

\item $D_y(g)=-D_x(Q D_y(P))$,

\smallskip

\item $g\in \Ima(D_x)$,

\smallskip

\item If $(\rho,\sigma)\in \Dir(P)$ with $\rho\ne 0$ and $v_{\rho,\sigma}(\tilde J)=v_{\rho,\sigma}(x^{-1})$, then there is $\lambda\in K^{\times}$ such that
$$
y-\lambda x^{\sigma/\rho}\mid \ell_{\rho,\sigma}(P)\quad\text{and}\quad \ev_{\lambda x^{\sigma/\rho}}(\ell_{\rho,\sigma}(\tilde J))=0.
$$
\end{enumerate}
\end{theorem}

\begin{proof} Statement~(1) follows directly from the fact that $D_y(\tilde J)=[P,Q]$. From~(1) it follows that
$$
\int_y D_y(g)=D_x\left(-\int_y Q D_y(P)\right)\in \Ima(D_x).
$$
Therefore, in order to prove~(2), it suffices to verify that $g-\int_y D_y(g)\in \Ima(D_x)$. But this is true since $g-\int_y D_y(g)\in
K[x^{1/l},x^{-1/l}]$ and
$$
\Coef_{x^{-1}y^0}\left(g-\int_y D_y(g)\right)=\Coef_{x^{-1}y^0}\left(\tilde J-Q D_x(P) \right)=\Coef_{x^{-1}y^0}\left(\int_y [P,Q]\right)=0.
$$
Now we prove statement~(3). It is convenient to consider separately the cases $\rho>0$ and $\rho<0$. We only deal with the case $\rho>0$, since the other one is similar. By Lemma~\ref{forma exacta}, statements~(1) and~(2) guarantee that the differential form
$$
(\tilde J -Q D_x(P))dx-(Q D_y(P))dy
$$
is exact. So, $\tilde J -Q D_x(P)=D_x(H) $ and $-Q D_y(P)=D_y(H)$ for some $H\in \hat L^{(l)}_-$. Let $\eta$ and $\lambda$ be as in Proposition~\ref{Newton Puiseux}, so that $y-\lambda x^{\sigma/\rho}$ divides $\ell_{\rho,\sigma}(P)$. We will prove that $\ev_{\lambda x^{\sigma/\rho}}(\ell_{\rho,\sigma}(\tilde J)) =0$. A direct computation shows that
$$
\ev_{\lambda x^{\sigma/\rho}}(\ell_{\rho,\sigma}(\tilde J))=\Coef_{x^{-1}y^0}(\ev_{\eta}(\tilde J))x^{-1}.
$$
Now by equality~\eqref{derivada total},
\begin{align*}
D_x(\ev_{\eta}(H)) &=\ev_{\eta}(D_x(H))+\ev_{\eta}(D_y(H))D_x(\eta)\\
&=\ev_{\eta}(\tilde J)-\ev_{\eta}(Q)\bigl(\ev_{\eta}(D_x(P))+\ev_{\eta}(D_y(P))D_x(\eta)\bigr)\\
&=\ev_{\eta}(\tilde J)-\ev_{\eta}(Q)D_x(\ev_\eta(P)).
\end{align*}
Since $\ev_{\eta}(P)=0$, we arrive at $D_x(\ev_{\eta}(H)) = \ev_{\eta}(\tilde J)$, and so $\Coef_{x^{-1}y^0}(\ev_{\eta}(\tilde J))=0$.
\end{proof}

\begin{corollary}\label{cassier nouges} If $P,Q\in L$ and $[P,Q]=\mu\in K^{\times}$, then there is no edge of the Newton polygon of $P$ with slope~$1$.
\end{corollary}

\begin{proof} Let $\tilde J$ be as in Theorem~\ref{cassier nouges generalizado}. Since $P,Q\in L$ and $[P,Q]=\mu$, we have $\tilde J=\mu y$. Then,
$$
\ev_{\lambda x^{-1}}(\ell_{-1,1}(\tilde J))=\ev_{\lambda x^{-1}}(\ell_{1,-1}(\tilde J))=\lambda \mu x^{-1}\ne 0\qquad\text{for each $\lambda\ne 0$.}
$$
But, by Theorem~\ref{cassier nouges generalizado}(3), if there is an edge with slope one, then there exists $\lambda\ne0$, such that
$$
\ev_{\lambda x^{-1}}(\ell_{-1,1}(\tilde J))=\ev_{\lambda x^{-1}}(\ell_{1,-1}(\tilde J))= 0,
$$
a contradiction.
\end{proof}

\section{Lower edges}\label{lower edges}
\setcounter{equation}{0}

Assume that the Jacobian Conjecture is false. Let $(P,Q)$, $m$ and $n$ be as in~\cite{GGV}*{Corollary~5.21}. In particular, $(P,Q)$ is a minimal pair and a standard $(m,n)$-pair in $L$ (see the beginning of Section~4 of~\cite{GGV} and~\cite{GGV}*{Definition~4.3}). In that paper we study the possible edges of the convex hull of the support of $P$ with a corner above the main diagonal. Here we begin the study of the lower part. In order to carry out this task, we proceed as follows: Consider $(a,b)\in\mathds{N}_0\times \mathds{N}_0$ with $a>b>0$ and suppose that there exists $(\rho,\sigma)\in \Dir(P) \cap \hspace{0.7pt}](0,-1),(1,0)[$ satisfying
$$
(a,b)=\frac 1m\en_{\rho,\sigma}(P).
$$
For such $(a,b)$ we are going to prove that there exist $(\rho,\sigma)$-homogeneous elements $G,R\in L$ such that $R$ is not a monomial,
$$
(a,b)=\en_{\rho,\sigma}(R)\quad\text{and}\quad [G,R]= R^i\quad\text{for some $i\ge 2$}.
$$
This allows to discard as possible $(a,b)$ all the points for which such $G, R$ do not exist.

\begin{proposition}\label{restriccion de direcciones} Let $(P,Q)$ be a Jacobian pair in $L$, $(\rho,\sigma)\in \Dir(P)$ and
$(a',b')\coloneqq \en_{\rho,\sigma}(P)$. As\-sume that $a'>b'>0$ and $(0,-1)<(\rho,\sigma)<(1,0)$. Then $(\rho,\sigma)<(1,-1)$.
\end{proposition}

\begin{proof} By Corollary~\ref{cassier nouges} we know that $(\rho,\sigma)\ne (1,-1)$. Suppose that $(\rho,\sigma) \in \hspace{2pt}](1,-1),(1,0)[ \hspace{2pt}$, which means that $\rho>-\sigma>0$. Since
$$
(\rho,\sigma)\in \mathfrak{V}_{>0} \quad\text{and}\quad v_{\rho,\sigma}(P)\ge v_{\rho,\sigma}(a',b') = \rho a'+\sigma b'>(\rho +\sigma) b'\ge \rho +\sigma >0,
$$
we can apply~\cite{GGV}*{Theorem~2.6}. Hence, there exists a $(\rho,\sigma)$-homogeneous polynomial $F$ such that
$$
(a',b') \sim \en_{\rho,\sigma}(F)\quad\text{or}\quad \en_{\rho,\sigma}(F) = (1,1),\quad v_{\rho,\sigma}(F) = \rho + \sigma\quad\text{and}\quad [F,\ell_{\rho,\sigma}(P)]= \ell_{\rho,\sigma}(P).
$$
Moreover, by~\cite{GGV}*{Remark~2.5} we know that $F$ is not a monomial. If $(a',b')\sim \en_{\rho,\sigma}(F)$, then there exists $\lambda>0$ such that $\en_{\rho,\sigma}(F) =\lambda (a',b')$. So
$$
\rho + \sigma = v_{\rho,\sigma}(F) = \rho \lambda a' + \lambda\sigma b' > \lambda b'(\rho+\sigma) \Longrightarrow 0<\lambda b' < 1,
$$
which is impossible, since $\lambda b' = v_{0,1}(\en_{\rho,\sigma}(F))\in \mathds{Z}$. Consequently, $\en_{\rho,\sigma}(F) = (1,1)$, and thus, by~\cite{GGV}*{Remark~1.8}, we have
$$
v_{0,1}(\st_{\rho,\sigma}(F)) < v_{0,1}(\en_{\rho,\sigma}(F)) = 1.
$$
Therefore $\st_{\rho,\sigma}(F) = (k,0)$ for some $k\in \mathds{Z}$, and so $\rho + \sigma = v_{\rho,\sigma}(\st_{\rho,\sigma}(F)) = \rho k$ which implies that $k>0$. But this leads to the contradiction  $\rho + \sigma = \rho k \ge \rho > \rho + \sigma$ and finishes the proof.
\end{proof}

In the rest of this section we assume that $(P,Q)$ is an $(m,n)$-pair in $L$ and we fix a direction $(\rho,\sigma)\in\hspace{2pt} ](0,-1),(1,-1)[$ of $P$ such that the point $(a,b)\coloneqq \frac 1m\en_{\rho,\sigma}(P)$ satisfies the in\-equal\-ities $a>b>0$.

\begin{proposition}\label{existe R} The valuations $v_{\rho,\sigma}(P)$ and $v_{\rho,\sigma}(Q)$ are greater than zero, and there exist a $(\rho,\sigma)$-homogeneous po\-ly\-nomial $R$ and $\lambda_P,\lambda_Q\in K^{\times}$ such that
\begin{equation}
\ell_{\rho,\sigma}(P) = \lambda_P R^m\quad\text{and}\quad \ell_{\rho,\sigma}(Q) = \lambda_Q R^n.\label{potencias de R}
\end{equation}
\end{proposition}

\begin{proof} Let $t_x(P)\coloneqq  \max \{0,\deg_x (P(0,x))\}$, where by definition $\deg_x(0)= -\infty$. By~\cite{vdE}*{Theorem~10.2.6}, we know that $t_x(P)>0$. Therefore $v_{\rho,\sigma}(P)\ge \rho t_x(P)>0$ and similarly $v_{\rho,\sigma}(Q)>0$. Hence,
$$
v_{\rho,\sigma}(P)+v_{\rho,\sigma}(Q)-(\rho+\sigma)>0
$$
since $\rho+\sigma<0$, and so, by~\cite{GGV}*{Proposition~1.13} we obtain that $[\ell_{\rho,\sigma}(P),\ell_{\rho,\sigma}(Q)]=0$. Moreover, by~\cite{vdE}*{Theorem~10.2.1}, we have
$$
\frac{v_{\rho,\sigma}(P)}{v_{\rho,\sigma}(Q)} = \frac{v_{1,1}(P)}{v_{1,1}(Q)} = \frac{m}{n}.
$$
Applying now~\cite{GGV}*{Proposition~2.1(2b)}, we finish the proof.
\end{proof}

\begin{remark}\label{en de R} From the first equality in \eqref{potencias de R} it follows immediately that $R$ is not a monomial and $\en_{\rho,\sigma}(R)=(a,b)$ .
\end{remark}

\begin{corollary}\label{R0} There exist $\lambda\in K^{\times}$ and a $(\rho,\sigma)$-homogeneous $R_0\in L$ such that $\ell_{\rho,\sigma}(P) = \lambda R_0^k$ with $k$ maximum (hence $m\mid k$ and we can assume that~$R=R_0^{k/m}$).
\end{corollary}

\begin{proposition}\label{corchete que da potencia de R0} There exist $G_0\in L$, $s\in \mathds{N}$ and $\mu\in K^{\times}$ such that
$$
[\ell_{\rho,\sigma}(G_0),\ell_{\rho,\sigma}(P)]=\mu R_0^s,
$$
where $R_0$ is as in Corollary~\ref{R0}.
\end{proposition}

\begin{proof} By~\cite{GGV}*{Lemma~2.2} we know that there exists $G_0\in L$, such that
$$
[\ell_{\rho,\sigma}(G_0),\ell_{\rho,\sigma}(P)]\ne 0 \quad\text{and}\quad [[\ell_{\rho,\sigma}(G_0),\ell_{\rho,\sigma}(P)],\ell_{\rho,\sigma}(P)]=0.
$$
Hence, by \cite{GGV}*{Proposition~2.1(2b)}, given $m_1,m_2\in \mathds{Z}$ such that $\frac{m_2}{m_1}=\frac{v_{\rho,\sigma}[\ell_{\rho,\sigma} (G_0),\ell_{\rho,\sigma}(P)]} {v_{\rho,\sigma}(P)}$ and $\gcd(m_1,m_2)=1$, there exist $R_1\in L$ and $\lambda_1,\lambda_2\in K^{\times}$ such that
$$
\ell_{\rho,\sigma}(P) = \lambda_1 R_1^{m_1} \quad\text{and}\quad [\ell_{\rho,\sigma}(G_0),\ell_{\rho,\sigma}(P)] = \lambda_2 R_1^{m_2}.
$$
Let $k$ be as in Corollary 2.4. Clearly $m_1\mid k$ and we can assume that $R_1 = R_0^{k/m_1}$. Therefore
$$
[\ell_{\rho,\sigma}(G_0),\ell_{\rho,\sigma}(P)] = \lambda_2 R_0^{m_2k/m_1}.
$$
So, the result is true with $s= m_2k/m_1$ and $\mu=\lambda_2$.
\end{proof}

\begin{corollary}\label{[G,R]=R^i} There exists a $(\rho,\sigma)$-homogeneous polynomial $G_1$ such that
\begin{equation}\label{eq [G,R]=R^i}
[G_1,R]= R^i \quad\text{for some $i\ge 0$,}
\end{equation}
where $R$ is as in Corollary~\ref{R0}.
\end{corollary}

\begin{proof} Let $G_0$ be as in Proposition~\ref{corchete que da potencia de R0}. By Propositions~\ref{existe R} and~\ref{corchete que da potencia de R0}, we have
$$
m R^{m-1}[\ell_{\rho,\sigma}(G_0),R]=[\ell_{\rho,\sigma}(G_0),R^m] =\frac{\mu}{\lambda_P} R_0^s,
$$
where $s$, $\lambda_P$ and $\mu$ are as in Propositions~\ref{existe R} and~\ref{corchete que da potencia de R0}. Let $k$ be as in Corollary~\ref{R0}. Since $m>1$ we can choose $j\in \mathds{N}$ such that $j\frac{k}{m}(m-1)\ge s$. Set
$$
G_1\coloneqq  \lambda_P\mu^{-1}m\ell_{\rho,\sigma}(G_0)R_0^t,
$$
where $t\coloneqq  j\frac{k}{m}(m-1)-s$. Clearly,
$$
m R^{m-1}[G_1,R]= \frac{\lambda_P}{\mu} R_0^t m^2 R^{m-1}[\ell_{\rho,\sigma}(G_0),R] = mR_0^{t+s}=mR_0^{j\frac{k}{m}(m-1)} = mR^{j(m-1)}.
$$
So, the equality in~\eqref{eq [G,R]=R^i} is true with $i=(j-1)(m-1)$.
\end{proof}

\begin{proposition}\label{desigualdad} Let $R$ be as in Corollary~\ref{R0}. The point $(a,b)$ satisfies
$$
v_{\rho,\sigma}(1,0)\le v_{\rho,\sigma}(a,b)=\frac 1m v_{\rho,\sigma}(P)=v_{\rho,\sigma}(R).
$$
\end{proposition}

\begin{proof} Let $t$ be maximum such that $(t,0)\in \Supp(P)$. By \cite{vdE}*{Proposition~10.2.6} we know that~$t$ exists and that $t>0$. By \cite{vdE}*{Theorem~10.2.1} we also know that $m\mid t$, so that
$$
\frac 1m v_{\rho,\sigma}(P)\ge \frac 1m v_{\rho,\sigma}(t,0)\ge v_{\rho,\sigma}(1,0),
$$
as desired.
\end{proof}

From Proposition~\ref{desigualdad} it follows that $(a,b)$ can be not of the form $(b+1,b)$ with $b>0$. In fact, the inequality
$$
v_{\rho,\sigma}(1,0)\le v_{\rho,\sigma}(b+1,b)
$$
implies that $b(\rho+\sigma)=v_{\rho,\sigma}(b,b)\ge 0$, which is impossible if $b>0$, since $\rho+\sigma<0$. On the other hand, we have the following result:

\begin{proposition}\label{ejemplo} Let $(a_0,b_0)\in\mathds{N}_0\times \mathds{N}_0$. If $(a_0,b_0)=(k,0)+j(2,1)$ with $k>0$ and $j>0$, then there exist $(1,-2)$-homogeneous polynomials $R$ and $G$ such that
$$
\st_{1,-2}(R) = (k,0),\quad \en_{1,-2}(R)=(a_0,b_0),\quad [G,R]= R^2
$$
and $R$ is not a monomial.
\end{proposition}

\begin{proof} Take $R\coloneqq x^k(1+x^2 y)^j$ and $G\coloneqq -\frac{1}{j+k}x^{k-1}(1+x^2 y)^{j+1}$.
\end{proof}

\begin{remark}\label{condicion para v sub rho sigma1} Let $(a_0,b_0)\in\mathds{N}_0\times \mathds{N}_0$ with $a_0>b_0>0$. If $v_{1,-2}(a_0,b_0)>0$, then we can apply Proposition~\ref{ejemplo} with $j:=b_0$ and $k:=a_0-2b_0$. So $(a_0,b_0)$ is a possible last lower corner (see Definition~\ref{Definicion pllc}).
\end{remark}

\section{Admissible chains}

In this section we introduce the notion of admissible chain, which encodes some of the properties of the lower part of $\frac{1}{m}H(P)$ (see Remark~\ref{parte inferior}). Here $P$ is the first component of an $(m,n)$-pair $(P,Q)$ in $L$. The main results are Theorem~\ref{condicion principal}, which shows we can assume $i=2$ in Corollary~\ref{[G,R]=R^i}, and our main technical result, Proposition~\ref{finitas direcciones}, which yields restrictions on the directions $(\rho,\sigma)$ that can occur for an $R$ as in Corollary~\ref{[G,R]=R^i} if you fix the starting point.

We begin by establishing some results that are useful for our purposes.

In the sequel $(\rho,\sigma)\in\mathfrak{V}\cap \hspace{0.7pt}](0,-1),(1,-1)[\hspace{0.7pt}$.

\begin{remark}\label{clausura integral} Assume $l\mid l_1$, $B\in L^{(l_1)}\setminus L^{(l)}$ and $A\in L^{(l)}\setminus \{0\}$. Then $AB\in L^{(l_1)}\setminus L^{(l)}$. Consequently, if $AB\in L^{(l)}$, $A\in L^{(l)}\setminus \{0\}$ and $B\in L^{(l_1)}$, then $B\in L^{(l)}$.
\end{remark}

\begin{lemma}\label{factores en jacobiano} Let $(\rho,\sigma)\in\mathfrak{V}\cap \hspace{0.7pt}](0,-1),(1,-1)[\hspace{0.7pt}$ and let $R,G_1\in L^{(l)}\setminus\{0\}$ be $(\rho,\sigma)$-homogeneous elements such that $[R,G_1]=R^i$ for an $i\in \mathds{N}_0$. Write $R= x^{u/\rho}r(z)$ and $G_1=x^{v/\rho}g(z)$, where $z\coloneqq x^{-\sigma/\rho}y$, $r$ and $g$ are univariate polynomials, $u\coloneqq v_{\rho,\sigma}(R)$ and $v\coloneqq  v_{\rho,\sigma}(G_1)$. The following facts hold:

\begin{enumerate}

\smallskip

\item We have
\begin{equation}\label{central1n}
\rho r^i=u g' r - v r'g.
\end{equation}

\smallskip

\item Let $h$ be a linear factor of $r$ and let $s$ and $t$ be the multiplicities of $h$ in $r$ and $g$ re\-spec\-tive\-ly. Write $r=h^s \tilde{r}$ and $g=h^t \tilde{g}$. Then
\begin{equation}\label{ecuacion de multiplicidad}
\rho h^{si}\tilde{r}^i = h^{t+s-1}\bigl((ut-vs)h'\tilde{g}\tilde{r} + h(u \tilde{g}' \tilde{r} - v\tilde{r}' \tilde{g})\bigr).
\end{equation}

\item If $i=2$, $\deg(g)\le \deg(r)+1$ and $\#\factors(r)>1$,  then there exists a linear factor~$h$ of~$r$ such that $ut-vs=0$.

\end{enumerate}

\end{lemma}

\begin{proof} (1)\enspace Write
$$
g(z)=\sum_{k=0}^{n_g}b_k z^k\quad\text{and}\quad r(z)=\sum_{k=0}^{n_r}a_k z^k.
$$
By~\cite{GGV}*{Proposition~1.13},
$$
iu=v_{\rho,\sigma}(R^i)=v_{\rho,\sigma}([R,G_1]) = v_{\rho,\sigma}(R) + v_{\rho,\sigma}(G_1) - \rho - \sigma = u + v - \rho - \sigma.
$$
Hence,
\begin{align*}
[x^{u/\rho}a_k z^k, x^{v/\rho}b_l z^l]&=\frac{a_k b_l}{\rho}x^{(u+v-\rho-\sigma)/\rho}z^{k+l-1}(lu-kv)\\
&= \frac {1}{\rho} x^{i u/\rho}[u l b_l z^{l-1} a_k z^k-v k a_k z^{k-1} b_l z^l].
\end{align*}
Using this we obtain that
$$
[R,G_1]=\frac 1{\rho} x^{iu/\rho}(u g'(z) r(z) - v r'(z)g(z)),
$$
which implies that $\rho r^i=u g' r - v r'g$, as we want,  because $R^i=[R,G_1]$.

\smallskip

\noindent (2)\enspace Since
$$
r' = s h^{s-1}h'\tilde{r} + h^s \tilde{r}'\quad\text{and} \quad g'=th^{t-1}h'\tilde{g}+h^t\tilde{g}',
$$
by statement~(1) we have
$$
\rho h^{si}\tilde{r}^i = h^{t+s-1}\bigl((ut-vs)h'\tilde{g}\tilde{r} + h(u \tilde{g}' \tilde{r} - v\tilde{r}' \tilde{g})\bigr),
$$
as we want.

\smallskip

\noindent (3)\enspace By statement~(2), if $i=2$ and $ut-vs\ne 0$ then $t=s+1$. If this is true for all factors, then
$$
\deg(g)\ge \deg(r)+\#\factors(r)>\deg(r)+1,
$$
a contradiction that concludes the proof.
\end{proof}

\begin{remark}\label{condicion para v sub rho sigma2} Let $(\rho,\sigma)\in\mathfrak{V}\cap \hspace{0.7pt}](0,-1),(1,-1)[\hspace{0.7pt}$ and let $(a,b)\in \mathds{N}_0\times \mathds{N}_0$ with $a>b>0$. If $a\le 2b$ and $v_{\rho,\sigma}(a,b)\ge \rho$, then
$$
v_{\rho,\sigma}(a,b)> -\rho-\sigma.
$$
In fact,
$$
(2\rho+\sigma)(a-b)=(\rho+\sigma)(a-2b)+v_{\rho,\sigma}(a,b)\ge (\rho+\sigma)(a-2b) + \rho> 0
$$
because $\rho+\sigma<0$ and $\rho>0$. So, $2\rho+\sigma>0$ since $a>b$, and hence,
$$
v_{\rho,\sigma}(a,b)\ge \rho = 2\rho+\sigma - (\rho+\sigma)> -\rho-\sigma,
$$
as we want. So, the condition $(n-1)v_{\rho,\sigma}(a,b)+\rho+\sigma\ge 0$ in Theorem~\ref{reduccion de j} is satisfied with $n=2$.
\end{remark}

\begin{theorem}\label{reduccion de j} Let $(\rho,\sigma)\in\mathfrak{V}\cap \hspace{0.7pt}](0,-1),(1,-1)[\hspace{0.7pt}$, let $i,n\in \mathds{N}_0$ with $n\ge 2$ and let $R,G_1\in L^{(l)}$ be $(\rho,\sigma)$-homogeneous elements. If $[R,G_1]=R^i$ and $(n-1)v_{\rho,\sigma}(R)+\rho+\sigma\ge 0$, then there exists a $(\rho,\sigma)$-homogeneous element $G_2\in L^{(l)}$ such that $[R,G_2]=R^n$. Moreover, if $R,G_1\in L$, then $G_2\in L$.
\end{theorem}

\begin{proof} If $i\in \{0,\dots,n\}$, then $G_2\coloneqq G_1R^{n-i}$ works. Else take $G_2\coloneqq \frac{G_1}{R^{i-n}}$ (in the field $\mathds{{}_{l_1}Q}$ of quotients of $L^{(l_1)})$, where $l_1:=\lcm(\rho,l)$. Since clearly $[R,G_1]=R^n$ (the Jacobian operator can be extended in a natural way to $\mathds{{}_{l_1}Q}$), we only must check that $G_2\in L^{(l)}$. By Remark~\ref{clausura integral} with $A=R^{i-n}$ and $B=G_2$, in order to carry out this task it suffices to prove that $G_2\in L^{(l_1)}$. Set $u\coloneqq v_{\rho,\sigma}(R)$ and $v\coloneqq  v_{\rho,\sigma}(G_1)$. By~\cite{GGV}*{Proposition~1.13} we have
\begin{equation}\label{igualdad de vs}
u+v-(\rho+\sigma)=i u.
\end{equation}
Therefore
\begin{equation}\label{otra igualdad de vs}
v=(i-n)u+((n-1)u+\rho+\sigma).
\end{equation}
Since $\rho+\sigma<0$ and, by hypothesis, $(n-1)u+\rho+\sigma\ge 0$, it follows that necessarily $u>0$. Hence equality~\eqref{otra igualdad de vs} implies that $v>0$, because $i-n>0$ and $(n-1)u+\rho+\sigma\ge 0$. Since $R$ and $G_1$ are $(\rho,\sigma)$-homogeneous there exist univariate polynomials $r$ and $g$ such that $R= x^{u/\rho}r(z)$ and $G_1=x^{v/\rho}g(z)$, where $z\coloneqq  x^{-\sigma/\rho}y$. Let $h$ be a linear factor of $r$ and let $s$ and $t$ be the multiplicities of $h$ in $r$ and $g$ respectively.

We claim that $t\ge s(i-n)$. Write $r=h^s \tilde{r}$ and $g=h^t \tilde{g}$. By Lemma~\ref{factores en jacobiano}(2) we know that
\begin{equation}\label{ecuacion de multiplicidad1}
\rho h^{si}\tilde{r}^i = h^{t+s-1}\bigl((ut-vs)h'\tilde{g}\tilde{r} + h(u \tilde{g}' \tilde{r} - v\tilde{r}' \tilde{g})\bigr).
\end{equation}
If $ut-vs=0$, then by equality~\eqref{igualdad de vs},
$$
t=\frac{vs}{u}=\frac{s}{u} \bigl((i-1)u+\rho+\sigma\bigr)=s\Bigr(i-1+\frac{\rho+\sigma}{u}\Bigr)\ge s(i-n),
$$
because $\frac{\rho+\sigma}{u}\ge 1-n$ by hypothesis, and the claim is true. On the other hand, if $ut-vs\ne 0$, then comparing the multiplicities of $h$ in~\eqref{ecuacion de multiplicidad1} we obtain that $si=t+s-1$. But then
$$
t=s(i-1)+1\ge s(i-n),
$$
which proves that the claim is also true in this case.

\smallskip

By the claim there exists $f(z)\in K[z]$ such that $g(z)=r(z)^{i-n}f(z)$, which implies that
$$
G_2=\frac{G_1}{R^{i-n}}=\frac{x^{v/\rho}g(z)}{(x^{u/\rho}r(z))^{i-n}}=x^{(v-(i-n)u)/\rho}f(z)\in L^{(l_1)},
$$
as desired.

\smallskip

Assume now that $R,G_1\in L$. It remains to check that $v_{-1,0}(G_2)\le 0$. For this we compute
$$
v-(i-n)u=(i-1)u+\rho+\sigma-(i-n)u=(n-1)u+\rho+\sigma\ge 0,
$$
which yields $v_{-1,0}(G_2)=v_{-1,0}(x^{v-(i-n)u)/\rho}f(z))\le 0$, since $v_{-1,0}(z)<0$.
\end{proof}

\begin{remark} Let $(\rho,\sigma)\in\mathfrak{V}\cap \hspace{0.7pt}](0,-1),(1,-1)[\hspace{0.7pt}$ and let $(a,b)\in \mathds{N}_0\times \mathds{N}_0$ with $a>b>0$. By Proposition~\ref{ejemplo}, if $a>2b$, then there exist $(\rho,\sigma)$-homogeneous polynomials $R$ and $G$ such that $R$ is not a mono\-mial and
$$
(a,b)=\en_{\rho,\sigma}(R)\quad\text{and}\quad [G,R]= R^2.
$$
\end{remark}

Next we prove that the previous result holds under a different condition

\begin{theorem}\label{condicion principal} Let $(\rho,\sigma)\in\mathfrak{V}\cap \hspace{0.7pt}](0,-1),(1,-1)[\hspace{0.7pt}$ and let $(a,b)\in \mathds{N}_0\times \mathds{N}_0$ with $a>b>0$. If $v_{\rho,\sigma}(a,b)\ge \rho$ and there exist $(\rho,\sigma)$-ho\-mogeneous elements $R,G_1\in L^{(l)}\setminus\{0\}$
such that
\begin{equation*}
(a,b)=\en_{\rho,\sigma}(R)\quad \text{and}\quad [G_1,R]= R^i\quad \text{for some $i\ge 0$,}
\end{equation*}
then there exists a $(\rho,\sigma)$-homogeneous polynomial $G$ such that $[G,R]= R^2$.
\end{theorem}

\begin{proof} By Remark~\ref{condicion para v sub rho sigma2} we have $v_{\rho,\sigma}(a,b)+\rho+\sigma\ge 0$. So, the hypothesis of Theorem~\ref{reduccion de j} are fulfilled for $n=2$, and applying it we obtain a $(\rho,\sigma)$-homogeneous polynomial $G$ such that $[G,R]= R^2$, as we want.
\end{proof}

\begin{lemma}\label{maximo grado} Let $(\rho,\sigma)\in\mathfrak{V}\cap \hspace{0.7pt}](0,-1),(1,-1)[\hspace{0.7pt}$ and let $R,G\in L^{(l)}$ be $(\rho,\sigma)$-homogeneous elements such that $[G,R]= R^i$ with $i\in\mathds{N}$ and $v_{\rho,\sigma}(R)>0$. Then
\begin{equation}\label{eq1}
v_{0,1}(\en_{\rho,\sigma}(G))\le (i-1)v_{0,1}(\en_{\rho,\sigma}(R))+1.
\end{equation}
\end{lemma}

\begin{proof} By~\cite{GGV}*{Proposition~2.4}, either $\en_{\rho,\sigma}(G)=(i-1) \en_{\rho,\sigma}(R)+(1,1)$ or $\en_{\rho,\sigma}(G) \sim \en_{\rho,\sigma}(R)$. In the first case, clearly
\begin{equation*}
v_{0,1}(\en_{\rho,\sigma}(G))= (i-1)v_{0,1}(\en_{\rho,\sigma}(R))+1.
\end{equation*}
Assume that $\en_{\rho,\sigma}(G)\sim \en_{\rho,\sigma}(R)$ and set $u\coloneqq v_{\rho,\sigma}(R)$ and $v\coloneqq v_{\rho,\sigma}(G)$. By~\cite{GGV}*{Proposition~1.13} we have $v=u(i-1)+\rho+\sigma$. So
$$
\en_{\rho,\sigma}(G)=\frac{v}{u} \en_{\rho,\sigma}(R)= (i-1)\en_{\rho,\sigma}(R)+\frac{\rho+\sigma}{u} \en_{\rho,\sigma}(R).
$$
Consequently
$$
v_{0,1}(\en_{\rho,\sigma}(G))= (i-1)v_{0,1}(\en_{\rho,\sigma}(R))+\frac{\rho+\sigma}{u}v_{0,1}(\en_{\rho,\sigma}(R)).
$$
Since $\frac{\rho+\sigma}{u}v_{0,1}(\en_{\rho,\sigma}(R)) \le 0$ because $u>0$, $\rho+\sigma<0$ and $v_{0,1}(\en_{\rho,\sigma}(R))\ge 0$, the the inequality~\eqref{eq1} also holds in this case.
\end{proof}

\begin{remark}\label{multiplicidad de la potencia} Let $f,\ov f\in K[x]$ be polynomials.
If $f(x)=\ov f(x^n)$, then $\lambda$ is a root of $f$ if and only if $\lambda^n$ is a root of
$\ov f$. Moreover, if $\lambda\ne 0$, then the multiplicity $m_\lambda$ of $\lambda$ in $f$ is the same as the
multiplicity $\ov m_{\lambda^n}$ of $\lambda^n$ in $\ov f$.
\end{remark}

\begin{remark}\label{gap}
Let $a,b,l\in\mathds{N}$. Set
$$
d\coloneqq\gcd(a,bl),\quad (\rho,\sigma)\coloneqq \left(\frac{bl}{d},-\frac{a}{d}\right)\quad\text{and}\quad \gap(\rho,l):=\frac{\rho}{\gcd(\rho,l)}.
$$
Note that, since $\gcd(a,b)=\gcd(b,d)$, we have
\begin{equation}\label{otra formula para gap}
\gap(\rho,l) = \frac{bl}{\gcd(bl,dl)} = \frac{b}{\gcd(b,d)} = \frac{b}{\gcd(a,b)}.
\end{equation}
Assume now that $R\in L^{(l)}$ is a $(\rho,\sigma)$-homogeneous element which is not a monomial, write
$$
R=x^{u/l}y^v f(z),\quad\text{where $z\coloneqq x^{-\sigma/\rho}y$ and $f(z) = \sum a_i z^i$ with $a_0\ne 0$,}
$$
and set $(a/l,b) \coloneqq \en_{\rho,\sigma}(R)-\st_{\rho,\sigma}(R)$. Then $a_i\ne 0$ implies $\gap(\rho,l)\mid i$. In fact, in that case we have $(u/l,v)+i(-\sigma/\rho,1)\in \frac 1l\mathds{Z}\times \mathds{N}_0$, which implies $-i\frac{\sigma}{\rho}\in\frac 1l \mathds{Z}$. Hence $\rho\mid il$, and so $\gap(\rho,l)\mid i$.
\end{remark}

\begin{notation}\label{starting y end}
Let $(\rho,\sigma)\in\mathfrak{V}\cap \hspace{0.7pt}](0,-1),(1,-1)[\hspace{0.7pt}$ and let $R\in L^{(l)}$ be a $(\rho,\sigma)$-ho\-mo\-ge\-neous element which is not a monomial. Let $\left(\frac{\upsilon_1}{l},\nu_1\right)\coloneqq  \en_{\rho,\sigma}(R) -\st_{\rho,\sigma}(R)$ and $\left(\frac{\upsilon_2}{l},\nu_2\right)\coloneqq \st_{\rho,\sigma}(R)$. In the sequel we set
$$
\quad N_1=N_1(R)\coloneqq \frac{\nu_1}{\gap(\rho,l)}\qquad\text{and}\qquad N_2=N_2(R)\coloneqq \gcd(\upsilon_2,\nu_2).
$$
Note that by Remark~\ref{gap} we have $N_1=\gcd(v_1,\nu_1)$.
\end{notation}

\begin{notation} For each $l\in \mathds{N}$ and each $(r/l,s)\in\frac{1}{l}\mathds{Z}\times\mathds{Z}\setminus\mathds{Z}(1,1)$, we let $\dir(r/l,s)$ denote the unique $(\rho,\sigma)\in\mathfrak{V}_{>0}$ such that $v_{\rho,\sigma}(r/l,s)=0$ (see the discussion below~\cite{GGV}*{Remark~3.1}).
\end{notation}

\begin{proposition}\label{finitas direcciones} Let $(\rho,\sigma)\in\mathfrak{V}\cap \hspace{0.7pt}](0,-1),(1,-1)[\hspace{0.7pt}$ and let $R,G\in L^{(l)}$ be $(\rho,\sigma)$-ho\-mo\-ge\-neous elements such that $R$ is not a monomial and $[G,R]= R^i$, where $i\in \mathds{N}$, and let $N_1$ and $N_2$ be as in Notation~\ref{starting y end}. Write $u\coloneqq v_{\rho,\sigma}(R)$ and $v\coloneqq v_{\rho,\sigma}(G)$. Assume that there exists $\ell>0$ such that $\ell u +\rho+\sigma>0$ (or, equivalently, that $u>0$). Write $R= x^{u/\rho}r(z)$ and $G=x^{v/\rho}g(z)$, where $r$ and $g$ are univariate polynomials and $z\coloneqq x^{-\sigma/\rho}y$. Then one of the following three cases occurs:
\begin{enumerate}

\smallskip

\item $\rho\mid l$ and $r=\xi h^j$ for some $\xi\in K^{\times}$, some linear polynomial $h\ne z$ and some $j\in \mathds{N}$.

\smallskip

\item There exist $\vartheta,t'\in \mathds{N}$ such that
$$
\qquad \vartheta\le N_1,\quad 0<t'<\ell \vartheta \quad\text{and}\quad (\rho,\sigma)=-\dir\left(t'\st_{\rho,\sigma}(R)+ \vartheta (1,1)\right).
$$
In this case there exists a linear factor of $r(z)$ with multiplicity $\vartheta$.

\smallskip

\item There exist $\vartheta,t'\in \mathds{N}$ such that
$$
\qquad \vartheta\mid N_2,\quad 0<t'<\ell \vartheta \quad\text{and}\quad (\rho,\sigma)= -\dir\left(t'\st_{\rho,\sigma}(R)+\vartheta (1,1)\right).
$$
In this case $\nu_2>0$.
\end{enumerate}
Moreover if $l=1$ and item~(1) occurs, then $v_{1,-2}(\en_{\rho,\sigma}(R))>0$.
\end{proposition}

\begin{proof} First note that
\begin{equation}\label{starting R no esta en la diagonal}
  \st_{\rho,\sigma}(R)\notin \mathds{N}_0(1,1),
\end{equation}
since $ \st_{\rho,\sigma}(R)=(n,n)$ implies $0<u=v_{\rho,\sigma}(\st_{\rho,\sigma}(R))=(\rho+\sigma)n\le 0$, which is impossible.
For each linear factor $h$ of $r$, we let $s$ and $t$ denote the multiplicities of $h$ in $r$ and $g$, res\-pectively. By equality~\eqref{ecuacion de multiplicidad} we know that
\begin{equation}\label{eq2}
t=s(i-1)+1\quad\text{or}\quad ut=vs,
\end{equation}
while by Lemma~\ref{maximo grado} we know that
$$
\deg(g)\le (i-1)\deg(r)+1.
$$
Therefore, if for all linear factors of $r$ the first equality in~\eqref{eq2} is satisfied, then there can be only one linear factor $h$ in $r$. Since $h$ is not a monomial (since $R$ is not), this implies that
$$
(u/\rho,0), (u/\rho-\sigma/\rho,1)\in\Supp(R)\subseteq \frac 1l\mathds{Z}\times \mathds{N}_0,
$$
which yields $\rho\mid l$, since $\gcd(\rho,\sigma)=1$. So, we are in case~(1).

Else there exists a factor $h=z-\lambda$ for which $ut=vs$. Now we will prove that if $\lambda\ne0$, then we are in case~(2). We set $t'\coloneqq s(i-1)-t$. Since $[G,R]= R^i$, by~\cite{GGV}*{Proposition~1.13} we have
\begin{equation}\label{ecuacion para a y b}
v=u(i-1)+\rho+\sigma.
\end{equation}
So
\begin{equation}\label{formula para t'}
t'=s(i-1)-t= s\frac{v}{u}-s\frac{\rho+\sigma}{u}-t= -s\frac{\rho+\sigma}{u}.
\end{equation}
Let $\ell$ be as in the statement. Since $\ell>-\frac{\rho+\sigma}{u}>0$ and $s>0$, from equality~\eqref{formula para t'} we obtain that $0<t'<\ell s$. Moreover,
\begin{equation}\label{paco}
t'\st_{\rho,\sigma}(R) + s (1,1)\notin\mathds{N}_0(1,1),
\end{equation}
since otherwise $\st_{\rho,\sigma}(R) = (n,n)$ for some $n\in \mathds{N_0}$, and so
$$
(\rho+\sigma)(\ell n+1) = \ell u + \rho + \sigma >0,
$$
which is impossible because $\rho + \sigma<0$. Combining~\eqref{paco} with the fact that, by equality~\eqref{ecuacion para a y b},
$$
0=u(i-1)-v+\rho+\sigma=u\frac{s(i-1)-t}{s}+\rho+\sigma=v_{\rho,\sigma}\left(\frac{t'}{s} \st_{\rho,\sigma}(R) +(1,1)\right),
$$
we conclude that
$$
(\rho,\sigma)=-\dir\left(t'\st_{\rho,\sigma}(R)+ s (1,1)\right).
$$
It remains to check that $s\le N_1$. Remember that $(\upsilon_2,\nu_2)\coloneqq \st_{\rho,\sigma}(R)$ and let $r_1$ be an univariate polynomial
such that $r(z)= z^{\nu_2}r_1(z)$. By Remark~\ref{gap}, since $R\in L^{(l)}$ and
$$
R = x^{u/\rho}r(z) = x^{u/\rho}z^{\nu_2}r_1(z) = x^{\upsilon_2 + \sigma\nu_2/\rho}x^{-\nu_2\sigma/\rho}y^{\nu_2} r_1(z) = x^{\upsilon_2} y^{\nu_2} r_1(z),
$$
there exists an univariate polynomial $\hat{r}_1$ such that $r_1(z)=\hat{r}_1(z^{\gap(\rho,l)})$. It follows that
$$
s=\mult_{r_1(z)}(\lambda)=\mult_{\hat r_1(z)}(\lambda^{\gap(\rho,l)})\le \deg(\hat{r}_1)=\frac{\deg(r_1)}{\gap(\rho,l)}=\frac{\nu_1}{\gap(\rho,l)}=N_1,
$$
where the second equality holds by Remark~\ref{multiplicidad de la potencia}. Setting in this case $\vartheta:=s$, we are in case~(2).

Now assume that there exists a factor $h=z-\lambda$ for which $ut=vs$ and that $\lambda=0$. Since equality~\eqref{ecuacion para a y b} is also true in this case, from the fact that $\rho+\sigma < 0$, it follows that
\begin{equation}\label{inecuacion para a y b}
v=(i-1)u+\rho+\sigma < (i-1)u.
\end{equation}
We assert that $\st_{\rho,\sigma}(G)\sim\st_{\rho,\sigma}(R)$. In fact, otherwise by~\cite{GGV}*{Proposition~2.4(1)},
$$
\st_{\rho,\sigma}(G)= (i-1)\st_{\rho,\sigma}(R)+(1,1),
$$
and so
$$
t=v_{0,1}(\st_{\rho,\sigma}(G))=(i-1)v_{0,1}(\st_{\rho,\sigma}(R))+1=s(i-1)+1,
$$
which, combined with the fact that $u>0$ by hypothesis, implies that
$$
v = \frac{ut}{s} = u(i-1)+\frac{u}{s}>u(i-1),
$$
contradicting~\eqref{inecuacion para a y b}. Write $\st_{\rho,\sigma}(G)=\frac{\mu_1}{\vartheta}\st_{\rho,\sigma}(R)$ for some coprime natural numbers~$\mu_1$ and~$\vartheta$. Recall from Notation~\ref{starting y end} that $N_2=\gcd(\upsilon_2,\nu_2)$, where $\left(\frac{\upsilon_2}{l},\nu_2\right) =\st_{\rho,\sigma}(R)$. It is clear that $\vartheta\mid N_2$. Set $t'\coloneqq \vartheta(i-1)-\mu_1$, and from $v=\frac{\mu_1}{\vartheta}u$ and~\eqref{ecuacion para a y b} we obtain
 $t'=-\frac{\rho+\sigma}{u}\vartheta$. Since $\ell>-\frac{\rho+\sigma}{u}>0$ and $\vartheta>0$, we have $0<t'<\ell \vartheta$. Moreover, again from $v=\frac{\mu_1}{\vartheta}u$ and~\eqref{ecuacion para a y b} it follows that
$$
0 = u(i-1) - v + \rho + \sigma = u\frac{\vartheta(i-1)-\mu_1}{\vartheta}+\rho + \sigma = v_{\rho,\sigma} \left(\frac{t'}{\vartheta} \st_{\rho,\sigma}(R) + (1,1)\right),
$$
which implies that $(\rho,\sigma) = -\dir\left(t'\st_{\rho,\sigma}(R) + \vartheta (1,1)\right)$, since $t' \st_{\rho,\sigma}(R) + \vartheta (1,1)\notin\mathds{N}_0(1,1)$. This shows that we are in the case~(3). Note that in this case $\nu_2=s>0$.

Finally if $l=1$ and statement~(1) is satisfied, then $\en_{\rho,\sigma}(R) = (u - j\sigma,j)$, and so
$$
v_{1,-2}(\en_{\rho,\sigma}(R))= u - j\sigma - 2j = u - j(\sigma + 2)>0,
$$
since $\sigma\le -2$.
\end{proof}

\begin{remark} Note that the from the equality $-\dir\left(t'\st_{\rho,\sigma}(R)+\vartheta (1,1)\right)$ in items~(2) and~(3) of Proposition~\ref{finitas direcciones} it follows that $t'= -\vartheta\frac{\rho+\sigma}{v_{\rho,\sigma}(R)}$.
\end{remark}

\begin{example}
A straightforward computation shows that for each $i,j,u\in \mathds{N}$ and each $\lambda\in K^{\times}$, the polynomials $R\coloneqq  x^u(x^{-\sigma}y -\lambda)^j$ and $G\coloneqq \varpi^{-1} x^v(x^{-\sigma}y - \lambda)^{j(i-1)+1}$, where $v\coloneqq u(i-1)+\sigma+1$ and $\varpi\coloneqq j((j-1)u(i-1) - j(\sigma+1))>0$, are in case~(1) of Proposition~\ref{finitas direcciones}. In Remark~\ref{caso 2} we will give a family of examples which are all in case~(2). Finally, an example in case~(3), with $i=2$, is given by $R\coloneqq 9 x^{14} y^8 (1 + x^8 y^5)$ and $G\coloneqq -x^7 y^4 (1 + x^8 y^5)^2$.
\end{example}

\subsection{Admissible chains}
An $(m,n)$-pair $(P,Q)$ in $L$ determines a chain of homogeneous polynomials $R_j$ together with a chain of segments (the lower part of $\frac{1}{m}H(P)$, where $H(P)$ is the Newton polygon of $P$). The last point of this chain is called the last lower corner of $(P,Q)$ (see Definition~\ref{llc}). Motivated by these facts, in this section we introduce the notion of admissible chain, consisting basically of a family of homogeneous polynomials and segments as above. This notion does not depend on the existence of a counterexample.  The final point of a such chain is named a possible last lower corner (see Definition~\ref{Definicion pllc}).

Suppose that the Jacobian conjecture is false and let
$$
B := \begin{cases}\infty & \text{if the Jacobian conjecture is true,}\\
\min\bigl(\gcd(v_{1,1}(P),v_{1,1}(Q))\bigr)&\text{if JC is false, where $(P,Q)$ runs on the counterexamples.}
\end{cases}
$$
The Jacobian conjecture is false if and only if the set of last lower corner's is not empty.

\smallskip

In this section we prove that some points in $\mathds{N}_0\times \mathds{N}_0$ are possible last lower corner's and that some other points are not.

\begin{definition}\label{cadena admisible} An {\em admissible chain of length $k\in\mathds{N}_0$} is a triple of families
$$
\mathfrak{C}=\bigl((C_j)_{j\in\{0,\dots,k\}},(R_j)_{j\in\{1,\dots,k\}},(\rho_j,\sigma_j)_{j\in\{1,\dots,k\}}\bigr),
$$
with $C_j\in \mathds{N}_0\times\mathds{N}_0$, $R_j\in L$ and $(\rho_j,\sigma_j)\in\mathfrak{V}\cap \hspace{0.7pt}](0,-1),(1,-1)[\hspace{0.7pt}$, such that

\begin{enumerate}

\smallskip

\item $C_0=(l,0)$ for some $l\in\mathds{N}$,

\smallskip

\item $(\rho_j,\sigma_j)>(\rho_{j-1},\sigma_{j-1})$, for $j\in \{2,\dots,k\}$,

\smallskip
\end{enumerate}
and the following facts hold for $j\in \{1,\dots,k\}$:
\begin{enumerate}[resume*]

\smallskip

\item $R_j$ is $(\rho_j,\sigma_j)$-homogeneous and is not a monomial,

\smallskip

\item $C_{j-1}=\st_{\rho_j,\sigma_j}(R_j)$,

\smallskip

\item $C_j=\en_{\rho_j,\sigma_j}(R_j)$,

\smallskip

\item $v_{1,-1}(C_j)>0$,

\smallskip

\item $v_{\rho_j,\sigma_j}(C_j)\ge\rho_j$,

\smallskip

\item there exist a $(\rho_j,\sigma_j)$-homogeneous $G_j\in L$ and $i_j\in\mathds{N}$ such that $[G_j,R_j]=R_j^{i_j}$.

\smallskip

\end{enumerate}
The point $C_k$ is called {\em the end point of $\mathfrak{C}$} and denoted $C_{\text{fin}}(\mathfrak{C})$ or sim\-ply~$C_{\text{fin}}$.
\end{definition}

\begin{remark}\label{los B_is crecen} For $0\le j\le k$ write $(a_j,b_j)\coloneqq C_j$. Since $(\rho_j,\sigma_j)\in \hspace{0.7pt}](0,-1),(1,-1) [\hspace{0.7pt}$, from items~(4) and~(5) it follows that $(a_i)_{0\le i\le k}$, $(b_i)_{0\le i\le k}$ and $(a_i-b_i)_{0\le i\le k}$ are increasing sequences. So, by item~(1), we conclude that $a_j,b_j,a_j-b_j\in \mathds{N}$ for all $1\le j\le k$.
\end{remark}

\begin{definition} \label{Definicion pllc} A point $(a,b)\in \mathds{N}_0\times\mathds{N}_0$ is called {\em a possible last lower corner} if there exists an admissible chain such that $C_{\text{fin}}=(a,b)$.
\end{definition}

\begin{remark} Note that for every $t\in\mathds{N}$, the point $(t,0)$ is a possible last lower corner, corresponding to a chain of length $0$.
\end{remark}

\begin{remark} Let $(a,b)\in\mathds{N}_0\times \mathds{N}_0$ with $a>b>0$. If $a-2b>0$, then by Remark~\ref{condicion para v sub rho sigma1} there exist $k>0$ and $j>0$ such that $(a,b)=(k,0)+j(2,1)$, which, by Proposition~\ref{ejemplo}, implies that there exist a $(1,-2)$-homogeneous polynomials $R$ such that
$$
\bigr(\bigl((k,0),(a,b)\bigr),R,(1,-2)\bigr),
$$
is an admissible chain of length~$1$. So, $(a,b)$ is a possible~last lower corner.
\end{remark}

\begin{remark}\label{parte inferior} Let $(P,Q)$ be an $(m,n)$-pair in $L$. By \cite{vdE}*{Proposition~10.2.6} there exists $(t,0)\in \Supp(P)$ with $t\in \mathds{N}$. Assume that $t$ is maximum satisfying this condition. If we run counterclockwise along several edges of the Newton polygon $H(P)$ of $P$, beginning in $(t,0)$ and stopping at a corner below the main diagonal of the plane, and we apply the homothety of center $0$ and ratio $\frac{1}{m}$ to each edge and each corner, then we obtain the families of edges and vertices of an admissible chain. In fact, by Proposition~\ref{existe R} each corner of $H(P)$ belongs to $m\mathds{N}_0\times m\mathds{N}_0$; by \cite{GGV}*{Definition~4.3} and Proposition~\ref{restriccion de direcciones} the directions of the edges belong to $\hspace{0.7pt}](0,-1),(1,-1)[\hspace{0.7pt}$; by~\cite{vdE}*{Proposition~10.2.6} condition~(1) of Definition~\ref{cadena admisible} is fulfilled; conditions~(2) and~(6) are clear; the existence of $R_j$'s and $G_j$'s satisfying conditions~(3), (4), (5) and~(8) follows from Proposition~\ref{existe R}, Remark~\ref{en de R} and Corollary~\ref{[G,R]=R^i}; and condition~(7) holds by Proposition~\ref{desigualdad}.
\end{remark}

\begin{definition}\label{llc} A point $(a,b)\in \mathds{N}_0\times\mathds{N}_0$ is called {\em a last lower corner} if there exists an admissible chain, obtained as in Remark~\ref{parte inferior} from a standard $(m,n)$-pair $(P,Q)$ in $L$, such that $C_{\text{fin}}=(a,b)$ and $m(a,b)$ is the last corner of $H(P)$ below the main diagonal of the plane. In this case we also will say that $(a,b)$ is {\em the last lower corner} of $(P,Q)$.
\end{definition}

\begin{remark}\label{Con Jac falsa} Let $(P,Q)$ be a standard $(m,n)$-pair in $L$ and let $(\rho,\sigma)\in\hspace{2pt} ](0,-1),(1,-1)[\cap \Dir(P)$ such that $(a,b)\coloneqq \frac 1m\en_{\rho,\sigma}(P)$ satisfies $a>b>0$. Then, by Remark~\ref{parte inferior} we know that $(a,b)$ is a possible last lower corner.
\end{remark}

\begin{remark}\label{Con Jac falsa starting triple} Let $(P,Q)$ be a standard $(m,n)$-pair in $L$ and let $(A_0,A_0',(\rho,\sigma))$ be the starting triple of $(P,Q)$ (see~\cite{GGV}*{Definition~6.2}). By~\cite{GGV}*{Remark~6.3} we know that $A_0'$ is a last lower corner.
\end{remark}

\begin{remark}\label{para prop 2.18} Let $(\rho,\sigma)\in\hspace{2pt} ](0,-1),(1,-1)[$ be a direction and let $R$ and $G$ be $(\rho,\sigma)$-ho\-mo\-ge\-neous polynomials such that $R$ is not a monomial and $[G,R]= R^i$, with $i\in \mathds{N}$. Arguing as in Re\-mark~\ref{condicion para v sub rho sigma2} we see that if $v_{\rho,\sigma}(R)\ge \rho$ and $v_{1,-2}(\en_{\rho,\sigma}(R))\le 0$, then
$$
v_{\rho,\sigma}(R)+\rho+\sigma>0.
$$
Since, moreover, it is clear that $v_{\rho,\sigma}(R)>0$ and $v_{1,-1}(\st_{\rho,\sigma}(R))\ne 0$, the hypothesis of Proposition~\ref{finitas direcciones} are fulfilled with $\ell=1$.
\end{remark}

\begin{lemma}\label{para la cota} Let $(\rho,\sigma)\in\hspace{2pt} ](0,-1),(1,-1)[$ be a direction and let $R$ and $G$ be $(\rho,\sigma)$-ho\-mo\-ge\-neous polynomials such that $R$ is not a monomial and $[G,R]= R^i$, where $i\in \mathds{N}$. Assume that $v_{\rho,\sigma}(R)\ge \rho$ and $v_{1,-2}(\en_{\rho,\sigma}(R))\le 0$, and write $(\alpha,\beta)\coloneqq  \st_{\rho,\sigma}(R)$ and $(\alpha',\beta')\coloneqq \en_{\rho,\sigma}(R)$. If $\beta<\alpha$ and $\beta\le (\alpha - \beta - 1)^2$, then  $\beta'<\alpha'$ and $\beta'\le (\alpha' - \beta' - 1)^2$.
\end{lemma}

\begin{proof} First note that $\alpha-\beta<\alpha'-\beta'$ and $\beta'<\alpha'$, because $(\rho,\sigma)\in\hspace{2pt} ](0,-1),(1,-1)[$ and $\beta<\alpha$. By Remark~\ref{para prop 2.18} the hypothesis of Proposition~\ref{finitas direcciones} are fulfilled with $\ell =1$. We will use freely its notations. Since $v_{1,-2}(\alpha',\beta')\le 0$, necessarily statements~(2) or~(3) of that proposition hold. In both cases we will use that,
\begin{equation}\label{auxiliar}
2(\alpha-\beta)(\alpha-\beta-1)= 2(\alpha-\beta-1)^2 + 2(\alpha-\beta-1)\ge 2\beta >\beta,
\end{equation}
since $\beta<\alpha$ and $\beta\le (\alpha - \beta - 1)^2$.

\smallskip

In the first case there exist $t',\vartheta\in \mathds{N}$ and $\zeta\in \mathds{Q}_{>0}$ such that
\begin{equation}\label{paquito}
(\alpha',\beta')=(\alpha,\beta)+\zeta\Bigl((\alpha,\beta)+\frac{\vartheta}{t'}(1,1)\Bigr).
\end{equation}
Moreover, $0<t'<\vartheta\le N_1$, where $N_1\coloneqq \gcd(\alpha'-\alpha,\beta'-\beta)$, and so
$$
\frac{\vartheta}{t'}\le N_1 \mid v_{1,-1}(\alpha'-\alpha,\beta'-\beta)=\zeta(\alpha-\beta)\le (\alpha-\beta)\zeta(\alpha-\beta),
$$
which combined with~\eqref{auxiliar} yields
$$
\beta +\frac{\vartheta}{t'}\le(\alpha - \beta)(\zeta(\alpha - \beta)+2(\alpha - \beta - 1)).
$$
Multiplying this inequality by $\zeta$ and adding the inequality $\beta\le (\alpha - \beta - 1)^2$, we obtain that
$$
\beta + \zeta\left(\beta + \frac{\vartheta}{t'}\right)\le(\zeta(\alpha - \beta) + \alpha - \beta - 1)^2.
$$
Combining this with~\eqref{paquito} we obtain $\beta'\le (\alpha' - \beta' - 1)^2$, as desired.

\smallskip

In the second case there exist $t',\vartheta\in \mathds{N}$ and $\zeta\in \mathds{Q}_{>0}$ such that
$$
(\alpha',\beta')=(\alpha,\beta)+\zeta\Bigl((\alpha,\beta)+\frac{\vartheta}{t'}(1,1)\Bigr),
$$
Moreover $\zeta (\alpha - \beta) = \alpha' - \beta' + \beta-\alpha\in\mathds{N}$  and $0<t'<\vartheta\mid N_2$, where $N_2\coloneqq \gcd(\alpha,\beta)$. So
$$
\frac{\vartheta}{t'}\le \vartheta\le N_2 \le \alpha - \beta \le (\alpha - \beta) \zeta(\alpha - \beta),
$$
which combined with~\eqref{auxiliar} yields
$$
\beta +\frac{\vartheta}{t'} \le (\alpha - \beta)(\zeta (\alpha - \beta) + 2(\alpha - \beta - 1)).
$$
Hence, arguing as above we obtain that $\beta' \le (\alpha' - \beta' - 1)^2$, concluding the proof.
\end{proof}

\begin{proposition}\label{cota} If $A\coloneqq (a,b)$ is a possible last lower corner, then $b<a$ and $b\le (a - b - 1)^2$.
\end{proposition}

\begin{proof} If $A=C_{\text{fin}}(\mathfrak{C})$ for an admissible chain $\mathfrak{C}$ of length~$0$, then $A=(t,0)$ for some~$t\in \mathds{N}$ and the result is obviously true. Assume that it is true for end points of  admissible chains of length~$k$ and that $A=C_{\text{fin}}(\mathfrak{C})$ for an admissible chain $\mathfrak{C}$ of length~$k+1$. If $v_{1,-2}(a,b)>0$, then $b<a - b$, so that $b\le a - b-1$, and hence $b\le (a - b-1)^2$.  On the other hand, if~$v_{1,-2}(a,b)\le 0$, then by the inductive hypothesis and Lemma~\ref{para la cota}, we also have $b<a$ and~$b\le (a - b - 1)^2$, as desired.
\end{proof}

\begin{corollary} For a fixed $g>0$ there are only finitely many points $A\in \mathds{N}_0\times \mathds{N}_0$ such that~$A$ is a possible last lower corner with $v_{1,-1}(A)=g$.
\end{corollary}

In the following remark we show that if $(a,b)\in \mathds{N}_0\times \mathds{N}_0$ satisfies $0<b<a$ and $b = (a-b-1)^2$, then $(a,b)$ is a possible last lower corner.

\begin{remark}\label{caso 2} Let $n\in\mathds{N}$. The following example attains the equality in the second bound in Proposition~\ref{cota}. Let $R=x(w+1)^n$, where $w=x^{n+1}y^n$. The triple
$$
\bigl(((1,0),(n^2+n+1,n^2)),(R),(n,-n-1)\bigr)
$$
is an admissible chain of length~$1$. In fact it is easy to check that $(n,-n-1)\in \hspace{2pt} ](0,-1),(1,-1)[$ and that conditions~(1) and (3)--(7) of Definition~\ref{cadena admisible} are fulfilled, condition~(2) is empty, and condition~(8) holds since the polynomial $G\coloneqq \frac{-1}{n+1}x^2 y(w+1)^{n-1}(w+n+1)$ satisfies $[G,R]=R^2$. Hence $(a,b) \coloneqq (n^2+n+1,n^2)$ is a possible last lower corner that satisfies $b = (a - b - 1)^2$. This example also shows that for any fixed $\lambda<1$, there is $(a,b)$, which is a possible last lower corner that satisfies $b>\lambda a$.
\end{remark}

In the proof of the following proposition we will use the concept of cross product introduced below~\cite{GGV}*{Notation~1.6} and we will use the property~(3.1) established at the beginning of Section~3 of~\cite{GGV}.

\begin{proposition}\label{casos imposibles} For each $\wp, n'\in\mathds{N}$ with $n'\ge 2$, the point $\wp (n',n'-1)$ is not a possible last lower corner.
\end{proposition}

\begin{proof} Assume by contradiction that there exists an admissible chain
$$
\bigl((C_j)_{j\in\{0,\dots,k\}},(R_j)_{j\in\{1,\dots,k\}},(\rho_j,\sigma_j)_{j\in\{1,\dots,k\}}\bigr)
$$
with $C_k\!=\!\wp(n',n'-1)$ and let $G_k$ be as in Definition~\ref{cadena admisible}(8). By the conditions of that de\-finition, the fact that $v_{1,-2}(C_k) = \wp n' - 2\wp (n'-1)<0$ and Remark~\ref{para prop 2.18}, the $(\rho_k,\sigma_k)$-ho\-mo\-ge\-neous polynomials $R_k$ and $G_k$ satisfy the hypothesis of Proposition~\ref{finitas direcciones} with $\ell=1$. Moreover $\rho_k > 1$, because otherwise $v_{\rho_k,\sigma_k}(R_k) = \wp (n'+(n'-1)\sigma_k)\le 0$. Hence case~(1) of that proposition can not occur. Now set $(\tilde{\rho},\tilde{\sigma})\coloneqq (n'-1,-n')$ and write $(a_{k-1},b_{k-1})\coloneqq C_{k-1}$. From
$$
0<\frac{1}{\wp } v_{\rho_k,\sigma_k}(C_k)=\rho_k n'+\sigma_k(n'-1) = (\tilde{\rho},\tilde{\sigma})\times(\rho_k,\sigma_k),
$$
we obtain that $(\tilde{\rho},\tilde{\sigma})<(\rho_k,\sigma_k)<(-\tilde{\rho},-\tilde{\sigma})$. Hence, by~\cite{GGV}*{Remark~1.8},
$$
v_{\tilde{\rho},\tilde{\sigma}}(a_{k-1},b_{k-1})= v_{\tilde{\rho},\tilde{\sigma}}(\st_{\rho_k,\sigma_k}(R_k)) > v_{\tilde{\rho},\tilde{\sigma}}(\en_{\rho_k,\sigma_k}(R))=0.
$$
Consequently, we have
\begin{equation}\label{desigualdad para N1 y N2}
N_1,N_2\le v_{\tilde{\rho},\tilde{\sigma}}(a_{k-1},b_{k-1}),
\end{equation}
where $N_1$ and $N_2$ are as in Notation~\ref{starting y end}. In fact,\ $N_2$ divides every integer combination of $a_{k-1}$ and $b_{k-1}$, and, in particular, it divides $v_{\tilde{\rho},\tilde{\sigma}}(a_{k-1},b_{k-1})>0$. Similarly, $N_1$ divides every integer combination of $\wp n'-a_{k-1}$ and $\wp (n'-1)-b_{k-1}$, and, in particular, it divides
$$
v_{\tilde{\rho},\tilde{\sigma}}(a_{k-1}-\wp n',b_{k-1}- \wp (n'-1))= v_{\tilde{\rho},\tilde{\sigma}}(a_{k-1},b_{k-1})>0.
$$
Suppose that we are in case~(2) of Proposition~\ref{finitas direcciones}. So there exist $\vartheta,t'\in \mathds{N}$ and $\lambda\in\mathds{Q}$ such that
\begin{equation}\label{desigualdades e igualdad}
0<t'< \vartheta \le N_1 \quad\text{and}\quad (a_{k-1},b_{k-1})+\frac{\vartheta}{t'}(1,1)=\lambda(-\sigma_k,\rho_k).
\end{equation}
Moreover, $\lambda>0$ since $0<v_{1,-1}(a_{k-1},b_{k-1})=-\lambda(\rho_k+\sigma_k)$ and $\rho_k+\sigma_k<0$. On one hand, by~\eqref{desigualdad para N1 y N2} and the inequalities in~\eqref{desigualdades e igualdad}, we have
$$
v_{\tilde{\rho},\tilde{\sigma}}\left((a_{k-1},b_{k-1})+\frac{\vartheta}{t'}(1,1)\right)=v_{\tilde{\rho},\tilde{\sigma}}(a_{k-1},b_{k-1})+ \frac{\vartheta}{t'}(n'-1-n') \ge v_{\tilde{\rho},\tilde{\sigma}}(a_{k-1},b_{k-1})-N_1\ge 0,
$$
and, on the other hand, by equality in~\eqref{desigualdades e igualdad} and the fact that $\lambda>0$, we have
\begin{align*}
v_{\tilde{\rho},\tilde{\sigma}}\left((a_{k-1},b_{k-1})+\frac{\vartheta}{t'}(1,1)\right)&=v_{\tilde{\rho},\tilde{\sigma}}(\lambda(-\sigma_k,\rho_k))= \lambda v_{\rho_k,\sigma_k}(\tilde{\sigma},-\tilde{\rho})\\
&=-\frac{\lambda}{\wp} v_{\rho_k,\sigma_k}(\wp(n',n'-1))=-\frac{\lambda}{\wp}v_{\rho_k,\sigma_k}(C_k)<0,
\end{align*}
which yields a contradiction. If we are in case~(3) of Proposition~\ref{finitas direcciones}, then replacing $N_1$ by $N_2$, the same argument works, finishing the proof.
\end{proof}

\begin{remark}
Propositions~\ref{cota} and~\ref{casos imposibles} give nice conditions for possible last lower corners, which are easy to understand. These conditions follow directly from Proposition~\ref{finitas direcciones}. Consequently, when we write an algorithm in order to compute the possible last lower corners, it suffices to consider the conditions of Proposition~\ref{finitas direcciones}, since then the conditions of Propositions~\ref{cota} and~\ref{casos imposibles} are automatically satisfied.
\end{remark}

\begin{remark} Let $A_0$ be as in the introduction. Clearly Proposition~\ref{casos imposibles} shows that $(2,1)$, $(3,2)$, $(6,3)$ and $(8,4)$ are not possible last lower corner's. Now we state without a formal proof, that for $A_0=(10,25)$ we necessarily have $A_0'=(2,1)$ and for $A_0=(14,35)$ we necessarily have $A_0'\in\{(6,3),(3,2)\}$. This allows to discard directly these corners $A_0$ of the list found in~\cite{H}*{Theorem~2.24}, which is included in those given in~\cite{GGV}*{Remark~7.9}.  The corners found in the two mentioned lists were also given without any proof, and were found by a computer search, but the algorithm
was not given explicitly. The same algorithm justifies the assertions above. Moreover, a straightforward argument shows that if the corner $(8,32)$ realizes an $A_0$, then we would obtain, after a transformation via an automorphism, the corner $A_0'=(8,4)$, which also is impossible. So our results permit to discard three of the corners of~\cite{H}*{Theorem~2.24}. We also can discard two of the infinite families of~\cite{H}*{Theorem~2.25}. In fact, the families $(5k+3,3k+2)$ and $(4k+3,k+1)$, corresponding to $A_0=(7,21)$, come from $A_0'=(2,1)$, which is impossible. This is the first time since Heitmann found the corners, that one of the infinite families can be discarded.

Furthermore, we can also discard some of the corners found in~\cite{GGV}*{Remark~7.9}, which were not found by Heitmann. Let $B_0$ and $B_1$ be as in that remark. The cases with $B_0=(6,15)$ and $B_1=(6,18+6k)$ where $18+6k$ is not a multiple of $30$, would lead to an $A_0'=(6,3)$ and can be discarded. Similarly $B_0=(8,28)$ and $B_1=(8,40)$ lead to $A_0'=(8,4)$, which is impossible; and $B_0=(9,21)$ and $B_1=(9,27)$ lead to an $A_0'=(9,6)$, which also is impossible.
\end{remark}

\paragraph*{Outlook}

A more computational paper is in preparation, where we will make explicit the algorithms that yields the corners of both lists (in fact the list of Heitmann is contained in the list of~\cite{GGV}), and we will explain the construction of the infinite families. We also will give an algorithm to determine all possible last lower corner with $v_{1,-1}(A_0')<N$ for some fixed $N$.

\end{document}